\documentclass[11pt,twoside]{amsart}
\usepackage{amsxtra}
\usepackage{color}
\usepackage{amsopn}
\usepackage{amsmath,amsthm,amssymb,arydshln}
\usepackage{mathrsfs,mathtools}
\usepackage{stmaryrd}
\usepackage{hyperref}
\usepackage{enumerate}
\usepackage{xcolor}
\usepackage{pifont}
\usepackage{adjustbox}
\usepackage{tikz}

\newtheorem{theorem}{Theorem}[section]
\newtheorem{corollary}[theorem]{Corollary}
\newtheorem{proposition}[theorem]{Proposition}

\newtheorem*{theorem*}{Theorem}
\theoremstyle{definition}
\newtheorem{definition}[theorem]{Definition}

\newtheorem{example}[theorem]{Example}
\newtheorem{remark}[theorem]{Remark}

\newcommand{\R}{{\mathbb R}}

\newcommand{\C}{{\mathbb C}}

\newcommand{\beq}{\begin{equation}}
\newcommand{\eeq}{\end{equation}}

\newcommand{\f}{\varphi}
\newcommand{\tf}{\widetilde{\f}}

\newcommand{\SU}{{\mathrm{SU}}}

\newcommand{\GL}{{\mathrm {GL}}}

\newcommand{\G}{{\mathrm G}}

\newcommand{\SL}{{\mathrm {SL}}}

\newcommand{\ddt}{\frac{\partial}{\partial t}}

\newcommand{\W}{\wedge}

\newcommand{\Ric}{{\rm Ric}}

\newcommand{\frg}{\mathfrak{g}}

\newcommand{\frn}{\mathfrak{n}}

\newcommand{\fre}{\mathfrak{e}}

\newcommand{\frs}{\mathfrak{s}}
\newcommand{\frz}{\mathfrak{z}}
\newcommand{\frr}{\mathfrak{r}}
\newcommand{\fru}{\mathfrak{u}}

\newcommand{\+}{{{\scriptscriptstyle +}}}

\newcommand{\st}{\ |\ }

\newcommand{\sst}{\scriptscriptstyle}

\newcommand{\jf}{\mathrm{j}_\f}

\numberwithin{equation}{section}

\title[A class of eternal solutions to the G$_2$-Laplacian flow]{A class of eternal solutions to the G$_{\mathbf2}$-Laplacian flow}
\author{Anna Fino and Alberto Raffero}
\address{Dipartimento di Matematica ``G. Peano'' \\ Universit\`a degli Studi di Torino\\ Via Carlo Alberto 10\\10123 Torino\\ Italy}
\email{annamaria.fino@unito.it, alberto.raffero@unito.it}
\thanks{The authors were supported by GNSAGA of INdAM}

\begin{document}
\begin{abstract}
We explicitly describe the solution of the G$_2$-Laplacian flow starting from an extremally Ricci-pinched closed G$_2$-structure on a compact 7-manifold  
and we investigate its properties. In particular, we show that the solution exists for all real times and that it remains extremally Ricci-pinched. 
This result holds more generally on any 7-manifold whenever the intrinsic torsion of the extremally Ricci-pinched G$_2$-structure has constant norm. 
We also discuss various examples. 
\end{abstract}
\maketitle

\section{Introduction} 
A $\G_2$-structure on a seven-dimensional smooth manifold $M$ is characterized by the existence of a 3-form $\f\in\Omega^3(M)$ satisfying a suitable non-degeneracy condition.   
Such a $3$-form gives rise to a Riemannian metric $g_{\f}$ and to a volume form $dV_{\f}$ on $M.$   

By \cite{FeGr}, the holonomy of $g_{\f}$ is contained in $\G_2$ if both $d\f$ and $d*_\f\f$ vanish, $*_\f$ being the Hodge operator of $g_\f$. 
On the other hand, when a Riemannian metric $g$ has $\mathrm{Hol}(g)\subseteq\G_2$, then there exists a unique $\G_2$-structure $\f$ satisfying $d\f=0$, $d*_{\f}\f=0$ and such that $g_{\f}=g$.
A $\G_2$-structure defined by a non-degenerate 3-form $\f$ which is both closed and co-closed is said to be {\em torsion-free} and the corresponding Riemannian metric $g_\f$ is Ricci-flat. 
A $\G_2$-structure $\f$ satisfying the less restrictive condition $d\f=0$ is called {\em closed}. 
In such a case, the intrinsic torsion can be identified with a unique 2-form $\tau$ 
such that $d*_\f\f=\tau\W\f$, and the scalar curvature of $g_\f$ is given by $-\frac12|\tau|_\f^2$, where $|\cdot|_\f$ denotes the norm induced by $g_\f$ (cf.~\cite{Bry}). 

Closed $\G_2$-structures with small torsion constitute the starting point in Joyce's construction of compact 7-manifolds with holonomy $\G_2$ \cite{Joy1}. 
Besides this and the glueing constructions \cite{CHNP,JoKa,Kov}, in recent years a lot of effort has been made in order to understand whether it is possible 
to obtain metrics with holonomy $\G_2$ using a geometric flow approach. 
So far, the main results in this direction have been obtained for the $\G_2$-{\em Laplacian flow} introduced by Bryant in \cite{Bry}: 
\[
\begin{cases}
\ddt \f(t) = \Delta_{\f(t)}\f(t),\\
d \f(t)=0,\\
\f(0)=\f.
\end{cases}
\]
Here, $\f$ is a given closed $\G_2$-structure and $\Delta_{\f(t)}$ denotes the Hodge Laplacian of $g_{\f(t)}$. 

Short-time existence and uniqueness of the solution of the Laplacian flow on a compact manifold were proved by Bryant and Xu in \cite{BrXu}, while  
the geometric and analytic properties of the flow were deeply investigated by Lotay and Wei in \cite{LotWei1,LotWei2,LotWei3}. 
In particular, they proved that the solution $\f(t)$ exists as long as the velocity of the flow $|\Delta_{\f(t)}\f(t)|_{\f(t)}$ remains bounded.  
It is still an open problem whether a bound on the scalar curvature is sufficient to obtain a long-time existence result (cf.~\cite{LotWei1}). 
Further aspects of the Laplacian flow were studied in \cite{FFM,FiYa,FiRa,HWY, Lau1,Lau2,Lin}. 

By \cite{Bry}, on a compact 7-manifold $M$ the Ricci tensor and the scalar curvature of the metric induced by a closed $\G_2$-structure $\f$ must satisfy the following inequality 
\[
\int_M\left[\mathrm{Scal}(g_\f)\right]^2dV_\f \leq 3 \int_M \left|\mathrm{Ric}(g_\f)\right|^2 dV_\f. 
\]
In particular, the metric $g_\f$ is Einstein if and only if the $\G_2$-structure is torsion-free (see also \cite{ClIv}). 
The above inequality reduces to an equality if and only if the intrinsic torsion form $\tau$ fulfills the equation $d\tau = \frac16\, |\tau|_\f^2\,\f +\frac16\,*_\f(\tau\W\tau).$ 
When this happens, the closed $\G_2$-structure is called {\em extremally Ricci-pinched} ({\em ERP} for short).  

Two examples of manifolds endowed with an ERP closed $\G_2$-structure were obtained by Bryant \cite{Bry} and Lauret \cite{Lau2}. 
Both can be described as simply connected solvable Lie groups endowed with a left-invariant ERP closed $\G_2$-structure. 
In the first case, the Lie group is not unimodular. Nevertheless, it admits a compact quotient by a torsion-free discrete subgroup of the full 
automorphism group of the $\G_2$-structure. In the second case, the Lie group is unimodular and the existence of a compact quotient has been recently proved in \cite{KaLa}. 
In \cite{Lau2}, Lauret proved that in both examples the ERP closed $\G_2$-structure $\f$ is a {\em steady Laplacian soliton}, i.e., it satisfies the equation 
\[
\Delta_\f\f = \mathcal{L}_X\f +\lambda\f,
\] 
for $\lambda=0$ and for some vector field $X$. 
General results on Laplacian solitons (cf.~e.g.~\cite[Sect.~9]{LotWei1}) allow one to conclude that the solution of the Laplacian flow starting from one of these ERP closed 
$\G_2$-structures is self-similar and {\em eternal}, i.e., it exists for all real times. 
By \cite{Lin}, compact Laplacian solitons with $\lambda=0$ are necessarily torsion-free. 
Thus, none of the above examples can descend to a steady Laplacian soliton on any compact quotient of the corresponding Lie group 
and, more generally, ERP closed $\G_2$-structures on compact manifolds cannot be steady Laplacian solitons.

In the present paper, we study the behaviour of the Laplacian flow starting from an ERP closed $\G_2$-structure in greater generality. 
Our main results are contained in Section \ref{LaplacianERP}. 
In Theorem \ref{MainThm}, we show that the solution of the Laplacian flow starting from an ERP closed $\G_2$-structure $\f$ whose intrinsic torsion form $\tau$ has constant norm is given by 
\[
\f(t) = \f + f(t)\,d\tau,
\]
where $f(t) = \frac{6}{|\tau|_\f^2}\left(\exp\left(\frac{|\tau|_\f^2}{6}\,t\right)-1\right)$. 
From this expression, we easily see that the solution exists for all real times. This result holds, in particular, when the closed $\G_2$-structure is ERP and the manifold $M$ is compact. 
To prove Theorem \ref{MainThm}, we first show some useful results on ERP closed G$_2$-structures in Proposition \ref{PrepProp}, and then we use them to show that the Laplacian flow 
we are considering is equivalent to a Cauchy problem for the function $f(t)$. 
The properties obtained in Proposition \ref{PrepProp} allow us to conclude that the solution $\f(t)$ is ERP with constant velocity for all $t\in\R$, 
and that the Ricci tensor of $g_{\f(t)}$ is constant along the flow.  
Finally, by backward uniqueness and real analyticity of the solution of the Laplacian flow on compact 7-manifolds \cite{LotWei1,LotWei3}, 
we conclude that a solution cannot become ERP in finite time unless it starts from an ERP closed $\G_2$-structure.  

In Section \ref{AsBe}, we study the asymptotic behaviour of the ERP solution $\f(t)$ in the compact case. 
In particular, we show that the volume of the manifold with respect to the Riemannian metric $g_{\f(t)}$ increases without bound as $t\rightarrow+\infty$, 
while it shrinks as $t\rightarrow-\infty$. 

In Section \ref{ExSect}, we review the two examples of ERP closed $\G_2$-structures mentioned above, and we discuss some related results. 
In Example \ref{1ParEx}, we show that Bryant's example belongs to a one-parameter family of inequivalent solvable Lie groups admitting a left-invariant ERP closed 
$\G_2$-structure, while in Theorem \ref{UnimUniq} we prove that a unimodular Lie group endowed with a left-invariant ERP closed $\G_2$-structure is isomorphic to Lauret's example.

\section{Preliminaries}

\subsection{Stable forms in dimension seven}\label{stableforms}
According to \cite{Hit}, a $k$-form on a real $n$-dimensional vector space $V$ is said to be {\em stable} if its $\GL(V)$-orbit is open in $\Lambda^k(V^*)$. 

In the present paper, we shall mainly deal with stable 3-forms in dimension seven. 
They can be characterized as follows. 
\begin{proposition}[\cite{Hit}]
Let $V$ be a seven-dimensional real vector space. Consider a 3-form $\phi\in\Lambda^3(V^*)$ and the symmetric bilinear map 
\[
b_\phi: V\times V\rightarrow\Lambda^7(V^*),\quad b_\phi(v,w) = \frac16\, \iota_v\phi\W\iota_w\phi\W\phi. 
\]
Then, $\phi$ is stable if and only if $\det(b_\phi)^{\sst1/9}\in\Lambda^7(V^*)$ is not zero. 
\end{proposition}

Given a stable 3-form $\phi$, the symmetric bilinear map 
\begin{equation}\label{gphirep}
g_\phi \coloneqq \det(b_\phi)^{\sst-1/9}\,b_\phi:V\times V\rightarrow\R
\end{equation}
is either positive definite or it has signature $(3,4)$. These conditions characterize the only two open $\GL(V)$-orbits contained in $\Lambda^3(V^*)$. 

We denote the open orbit of stable 3-forms for which \eqref{gphirep} is positive definite by $\Lambda^3_\+(V^*)$. 
It is well-known that the $\GL^\+(V)$-stabilizer of a 3-form $\phi\in\Lambda^3_\+(V^*)$ is isomorphic to the exceptional Lie group $\G_2$, 
and that there exists a basis $(e^1,\ldots,e^7)$ of $V^*$ for which  
\begin{equation}\label{G2adapted}
\phi = e^{123}+e^{145}+e^{167}+e^{246}-e^{257}-e^{347}-e^{356},
\end{equation}
$e^{ijk}$ being a shorthand for $e^i\W e^j \W e^k$. 

\subsection{Closed G$_\mathbf{2}$-structures} 
Let $M$ be a seven-dimensional smooth manifold and let $\Lambda^{3}_{\sst+}(T^*M)$ denote the open subbundle of $\Lambda^{3}(T^*M)$ 
whose fibre over each point $x\in M$ is given by $\Lambda^{3}_{\sst+}(T^*_xM)$. 

A $\G_2$-structure on $M,$ namely a $\G_2$-reduction of the frame bundle $FM\rightarrow M$, is characterized by the existence of a stable 3-form 
$\f\in\Omega^{3}_{\sst+}(M)\coloneqq \Gamma(\Lambda^{3}_{\sst+}(T^*M))$.  
This 3-form gives rise to a Riemannian metric $g_\f$ with volume form $dV_\f$ via the identity  
\[
g_\f(X,Y)\,dV_\f = \frac16\,\iota_{\sst X}\f\W\iota_{\sst Y}\f\W\f, 
\]
for all vector fields  $X,Y\in\mathfrak{X}(M)$. 
We denote by $*_\f$ the Hodge operator determined by $g_\f$, and by $|\cdot|_\f$ the pointwise norm induced by $g_\f$. 

It follows from the discussion in Section \ref{stableforms} that at each point $x$ of $M$ there exists a basis 
$\mathcal{B}^*=(e^1,\ldots,e^7)$ of the cotangent space $T^*_xM$ for which $\f|_x$ can be written as in \eqref{G2adapted}. 
We shall call $\mathcal{B}^*$ an {\em adapted basis} for the $\G_2$-structure $\f$. 

The intrinsic torsion of a $\G_2$-structure $\f$ can be identified with the covariant derivative of $\f$ with respect to the Levi Civita connection $\nabla^\f$ of $g_\f$. 
By \cite{FeGr}, a $\G_2$-structure $\f$ is {\em torsion-free}, i.e., $\nabla^\f\f\equiv0$, if and only if the 3-form $\f$ is both closed and coclosed. 

A $\G_2$-structure is said to be {\em closed} if the defining 3-form $\f$ satisfies the equation $d\f=0$. 
When this happens, the intrinsic torsion can be identified with a unique 2-form 
$\tau\in\Omega^{2}_{14}(M)\coloneqq\left\{\kappa\in\Omega^{2}(M)\st \kappa\W\f = -*_\f\kappa \right\}=\left\{\kappa\in\Omega^{2}(M)\st \kappa\W*_\f\f = 0 \right\}$ 
such that 
\[
d*_\f\f = \tau\W\f. 
\]
Clearly, the {\em intrinsic torsion form} $\tau$ vanishes identically if and only if the $\G_2$-structure is torsion-free. 
Notice that $\tau= d^*\f = -*_\f d *_\f\f$, thus it is coclosed and its exterior derivative coincides with the Hodge Laplacian $\Delta_\f\f = (dd^*+d^*d)\f = -d*_\f d *_\f \f$ of $\f$. 
Properties of closed $\G_2$-structures were investigated in \cite[sect.~4.6]{Bry} and in \cite{ClIv}. 

By \cite{HaLa}, the closed 3-form $\f$ defines a calibration on $M.$ 
An oriented three-dimensional submanifold of $M$ is called {\em associative} if it is calibrated by $\f$, while an oriented four-dimensional submanifold $N$ is called 
{\em coassociative} if $\f|_N\equiv0$ (see \cite[Sect.~IV]{HaLa} and \cite[Ch.~12]{Joy} for more details). 

By \cite{Bry}, the Ricci tensor and the scalar curvature of the Riemannian metric $g_\f$ induced by a $\G_2$-structure $\f$ can be expressed in terms of the intrinsic torsion.  
In particular, when $\f$ is closed the Ricci tensor has the following expression, 
\[
\mathrm{Ric}(g_\f) = \frac{1}{4}|\tau|_\f^2\, g_\f -\frac{1}{4}\,\mathrm{j}_\f\left(d\tau-\frac{1}{2}*_\f(\tau\W\tau)\right),
\]
where the map $\jf : \Omega^3(M) \rightarrow \mathcal{S}^2(M)$ is defined as follows
\[
\jf(\beta)(X,Y) = *_\f\left(\iota_{\sst X}\f\W\iota_{\sst Y}\f\W\beta\right),
\]
and the scalar curvature is given by
\[
\mathrm{Scal}(g_\f) = -\frac{1}{2}|\tau|_\f^2. 
\]

\subsection{The G$_{\mathbf2}$-Laplacian flow}
Consider a 7-manifold $M$ endowed with a closed $\G_2$-structure $\f$. The {\em Laplacian flow} starting from $\f$ is the initial value problem 
\begin{equation}\label{LapFlow}
\begin{cases}
\ddt \f(t) = \Delta_{\f(t)}\f(t),\\
d \f(t)=0,\\
\f(0)=\f.
\end{cases}
\end{equation}
This flow was introduced by Bryant in \cite{Bry} to study seven-dimensional manifolds admitting closed $\G_2$-structures. 
Short-time existence and uniqueness of the solution of \eqref{LapFlow} when $M$ is compact were proved in \cite{BrXu}. 
\begin{theorem}[\cite{BrXu}] 
Assume that  $M$ is compact. Then, the Laplacian flow \eqref{LapFlow} has a unique solution defined for a short time
$t\in[0,\varepsilon)$, with $\varepsilon$  depending on $\f$. 
\end{theorem}

\begin{remark}
The condition $d\f(t)=0$ implies that the solution of \eqref{LapFlow} must belong to the open set  
\[
[\f]_{\sst+}\coloneqq[\f]\cap\Omega^3_{\sst+}(M)
\]
in the de Rham cohomology class of $\f$ as long as it exists.  
\end{remark}

By \cite[Thm.~1.6]{LotWei1}, the solution $\f(t)$ exists as long as the velocity of the flow $\left| \Delta_{\f(t)}\f(t)\right|_{\f(t)}$ remains bounded. 
Moreover, if $\f(t)$ is defined on some interval $[0,T]$, then for each fixed time $t\in(0,T],$ $(M,\f(t),g_{\f(t)})$ is real analytic \cite{LotWei3}. 

Using general results on flows of $\G_2$-structures (see e.g.~\cite{Bry,Kar}), 
it is possible to check that the evolution equation of the Riemannian metric $g_{\f(t)}$ induced by a $\G_2$-structure $\f(t)$ evolving under the Laplacian flow is given by
\begin{equation}\label{metricevol}
\ddt g_{\f(t)} = -2\,\mathrm{Ric}(g_{\f(t)}) +\frac{|\tau(t)|_{\f(t)}^2}{6}\,g_{\f(t)}+\frac14\,\mathrm{j}_{\f(t)}\left(*_{\f(t)}(\tau(t)\W\tau(t))\right), 
\end{equation}
(see also \cite{LotWei1}), and the corresponding volume form $dV_{\f(t)}$ evolves as follows
\[
\ddt dV_{\f(t)} = \frac{|\tau(t)|^2_{\f(t)}}{3}\,dV_{\f(t)}.
\]
In particular, $dV_{\f(t)}$ is pointwise non-decreasing.

A solution of the Laplacian flow is said to be {\em self-similar} if it is of the form  
\[
\f(t) = \varrho(t)\, F_t^*\f,
\]
where $F_t\in{\rm Dif{}f}(M)$ and $\varrho(t)\in\R\smallsetminus\{0\}$ is a scaling factor. 
A standard argument allows one to show that the solution of the Laplacian flow is self-similar if and only if the initial datum $\f$ satisfies the equation
\begin{equation}\label{LapSol}
\Delta_\f\f = \mathcal{L}_X\f + \lambda\f,
\end{equation}
for some vector field $X$ on $M$ and some $\lambda\in\R$ (see e.g.~\cite{Lin,LotWei1}). 
In such a case, $\varrho(t)=\left(1+\frac23\lambda t\right)^{\sst3/2}$.  
A closed $\G_2$-structure for which \eqref{LapSol} holds is called a {\em Laplacian soliton}. 
Depending on the sign of $\lambda$, a Laplacian soliton is said to be {\em shrinking} ($\lambda<0$), {\em steady} ($\lambda=0$), or {\em expanding} ($\lambda>0$), 
and the corresponding self-similar solution exists on the maximal time interval  
$\left(-\infty,-\frac{3}{2\lambda}\right)$, $(-\infty,+\infty)$, $\left(-\frac{3}{2\lambda},+\infty\right)$, respectively.

\section{Extremally Ricci-pinched closed G$_2$-structures}
Let $M$ be a compact 7-manifold endowed with a closed $\G_2$-structure $\f$. 
It was proved independently in \cite{Bry} and \cite{ClIv} that the Riemannian metric $g_\f$ cannot be Einstein unless $\f$ is torsion-free. 
Moreover, by \cite{Bry} the Ricci tensor $\mathrm{Ric}(g_\f)$ and the scalar curvature $\mathrm{Scal}(g_\f)$ of $g_\f$ must satisfy the integral inequality 
\begin{equation}\label{intineq}
\int_M\left[\mathrm{Scal}(g_\f)\right]^2dV_\f \leq 3 \int_M \left|\mathrm{Ric}(g_\f)\right|^2 dV_\f, 
\end{equation}
and \eqref{intineq} reduces to an equality if and only if the intrinsic torsion form $\tau$ fulfills
\begin{equation}\label{ERPcond}
d\tau = \frac16\, |\tau|_\f^2\,\f +\frac16\,*_\f(\tau\W\tau).
\end{equation}
This motivates the following.
\begin{definition}[\cite{Bry}]
A closed $\G_2$-structure $\f$ whose intrinsic torsion form $\tau$ satisfies \eqref{ERPcond} is said to be {\em extremally Ricci-pinched} ({\em ERP} for short). 
\end{definition}

Useful properties of ERP closed $\G_2$-structures can be derived starting from \eqref{ERPcond} (cf.~\cite[Sect.~4.6]{Bry}). We summarize some of them in the next proposition. 
\begin{proposition}[\cite{Bry}]\label{ERPprop}
Let $M$ be a 7-manifold endowed with an ERP closed $\G_2$-structure $\f$ with intrinsic torsion form $\tau\in\Omega^{2}_{14}(M)$ not identically vanishing.  
If $M$ is compact, then $\tau$ has constant (non-zero) norm. More generally, if $|\tau|_\f$ is constant, then the following results hold 
\begin{enumerate}[i)]
\item\label{i} $\tau\W\tau\W\tau=0$;
\item\label{iii} $\tau\W\tau$ is a non-zero closed simple 4-form of constant norm;  
\item\label{iv} $*_\f(\tau\W\tau)$ is a non-zero closed simple 3-form of constant norm;
\item\label{v} by points \ref{iii}) and \ref{iv}), the tangent bundle of $M$ splits into the orthogonal direct sum of two integrable subbundles $TM=P\oplus Q$ with
\[
P\coloneqq\left\{X\in TM\st \iota_{\sst X}(\tau\W\tau)=0 \right\},\quad Q\coloneqq\left\{X\in TM\st \iota_{\sst X}*_\f(\tau\W\tau)=0 \right\}.
\]
Moreover, the $P$-leaves are associative submanifolds calibrated by $-|\tau|_\f^{-\sst2}*_\f(\tau\W\tau)$, while the $Q$-leaves are coassociative submanifolds calibrated by 
$-|\tau|_\f^{-\sst2}(\tau\W\tau)$;
\item\label{vi} the Ricci tensor of $g_\f$ is given by $\mathrm{Ric}(g_\f) = \frac{1}{12}\,\jf(*_\f(\tau\W\tau))= -\frac16\, |\tau|_\f^2\, g_\f|_P$. 
Hence, it is non-positive with eigenvalues $-\frac16\, |\tau|_\f^2$ of multiplicity three and $0$ of multiplicity four. 
\end{enumerate}
\end{proposition}

Further results on ERP closed $\G_2$-structures were obtained by Cleyton and Ivanov in \cite{ClIv2}. 
We shall recall one of them in Section \ref{ExSect}.

\begin{remark}\label{remnormdtau}
It follows from \eqref{ERPcond} and the identity $|\tau\W\tau|_\f^2=|\tau|^4_\f$ (cf.~\cite[(2.21)]{Bry}) that $d\tau$ has constant norm whenever $|\tau|_\f$ is constant, indeed
\[
|d\tau|_\f^2 = \frac{1}{6}|\tau|_\f^4.
\]
Moreover, as $\tau$ is coclosed and $d(\tau\W\tau)=0$, we have
\[
\Delta_\f\tau = -*_\f d *_\f d\tau = \frac{1}{6}|\tau|_\f^2\,\tau. 
\]
\end{remark}

\section{The Laplacian flow starting from an ERP closed G$_2$-structure}\label{LaplacianERP}
In this section, we prove the following result. 
\begin{theorem}\label{MainThm}
Let $M$ be a seven-dimensional manifold endowed with an ERP closed $\G_2$-structure $\f$ whose intrinsic torsion form $\tau$ has constant non-zero norm.  
Then, the solution of the Laplacian flow starting from $\f$ at $t=0$ is 
\begin{equation}\label{ERPsol}
\f(t) = \f + f(t)\,d\tau,
\end{equation}
where  
\[
f(t) = \frac{6}{|\tau|_\f^2}\left(\exp\left(\frac{|\tau|_\f^2}{6}\,t\right)-1\right). 
\]
In particular, $\f(t)$ is defined for all real times, it is ERP, and the corresponding intrinsic torsion form is given by
\[
\tau(t) = \exp\left(\frac{|\tau|_\f^2}{6}\,t\right)\tau.
\] 
Finally, the Ricci tensor of the Riemannian metric $g_{\f(t)}$ induced by $\f(t)$ is constant along the flow, i.e., $\mathrm{Ric}(g_{\f(t)})=\mathrm{Ric}(g_\f)$, 
and $g_{\f(t)}$ evolves as follows 
\[
\ddt g_{\f(t)} = \left.\frac{|\tau|^2_\f}{6}\,g_{\f(t)}\right|_Q. 
\]
\end{theorem} 

The proof of the first assertion in Theorem \ref{MainThm} consists in showing that the closed 3-form $\f(t)$ given by \eqref{ERPsol} defines a $\G_2$-structure and that the $\G_2$-Laplacian flow \eqref{LapFlow} for 
$\f(t)$ is equivalent to a Cauchy problem for the function $f(t)$. 
To this aim, it is useful to investigate the properties of the 3-form
\begin{equation}\label{ERPvariation}
\tf\coloneqq \f + a\,d\tau = \left(1+\frac16\, |\tau|_\f^2\, a \right)\f +\frac{a}{6}\,*_\f(\tau\W\tau),
\end{equation}
where $\f$ is an ERP closed $\G_2$-structure with intrinsic torsion form $\tau$ of constant norm, and $a$ is a real number. We collect them in the next result. 
\begin{proposition}\label{PrepProp}
Let $\f$ be an ERP closed $\G_2$-structure, assume that its intrinsic torsion form $\tau$ has constant norm, and consider the closed 3-form $\tf$ given by \eqref{ERPvariation}. 
Then, $\tf$ defines a closed $\G_2$-structure for all $a>-6|\tau|_\f^{-2}$.  
Whenever this happens, the following hold
\begin{enumerate}[1)]
\item\label{1} the $g_\f$-orthogonal decomposition $TM=P\oplus Q$ given in point \ref{v}) of Proposition \ref{ERPprop} is also $g_{\tf}$-orthogonal. 
Moreover, $\left.g_{\tf}\right|_P = \left.g_\f\right|_P$, and  $\left.g_{\tf}\right|_Q = \left(1+\frac16|\tau|_\f^2a\right) g_\f|_Q$;
\smallskip
\item\label{2} the volume form induced by $\tf$ is $dV_{\tf} = \left(1+\frac16|\tau|_\f^2a\right)^2dV_\f$;
\smallskip
\item\label{3} the Hodge dual of $\tf$ is $*_{\tf}\tf = \left(1+\frac16|\tau|_\f^2a\right)\left(*_\f\f -\frac{a}{6}\,\tau\W\tau\right)$;
\smallskip
\item\label{4} the intrinsic torsion form of $\tf$ is given by $\widetilde{\tau} = \left(1+\frac16|\tau|_\f^2a\right)\tau$, 
it satisfies the identity $*_{\tf}(\widetilde\tau\W\widetilde\tau) = *_\f(\tau\W\tau)$,   
and its $g_{\tf}$-norm coincides with the $g_\f$-norm of $\tau$, i.e., $|\widetilde\tau|_{\tf} = |\tau|_\f$. 
Consequently, $\mathrm{Scal}(g_{\tf}) = \mathrm{Scal}(g_\f)$; 
\smallskip
\item\label{5} the closed $\G_2$-structure $\tf$ is ERP and $\mathrm{Ric}(g_{\tf}) = \mathrm{Ric}(g_\f)$. 
\end{enumerate}
\end{proposition}
\begin{proof}
We begin showing that $\tf$ is stable for all $a\neq-6|\tau|_\f^{-2}$ and that it defines a $\G_2$-structure for all $a>-6|\tau|_\f^{-2}$. 
What we need to prove is that $\tf|_x\in\Lambda^3_{\sst+}(T_xM)$ for all $x\in M.$ 
Let us fix a basis $(e_1,\ldots,e_7)$ of $T_xM$ whose dual basis $(e^1,\ldots,e^7)$ of $T^*_xM$ is an adapted basis for the closed $\G_2$-structure $\f$. 
Then, we have
\[
\f|_x = e^{123}+e^{145}+e^{167}+e^{246}-e^{257}-e^{347}-e^{356},
\]
and $g_{\f}|_x = \sum_{i=1}^7(e^i)^2$. By point \ref{v}) of Proposition \ref{ERPprop}, we know that $T_xM=P_x\oplus Q_x$.  
Without loss of generality, we may assume that $P_x = \langle e_1,e_2,e_3\rangle$ and $Q_x=\langle e_4,e_5,e_6,e_7\rangle$. 
Since $P=\mathrm{ker}(\tau\W\tau)$, the simple 4-form $\tau\W\tau|_x$ must be proportional to $e^{4567}$. This implies that $\tau|_x \in\Lambda^2(Q_x^*)$. 
Consequently, as $\tau\W*_\f\f=0$, there exist some real numbers $c_1,c_2,c_3$ for which
\[
\tau|_x = c_1\left(e^{45}-e^{67}\right) +c_2\left(e^{46}+e^{57}\right) +c_3\left(e^{47}-e^{56}\right), 
\]
 In particular, $(\tau\W\tau)|_x = -|\tau|_\f^2\,e^{4567}$ and $*_\f(\tau\W\tau)|_x = -|\tau|_\f^2\,e^{123}$. 

We can now compute the symmetric bilinear map $b_{\tf}|_x:T_xM\times T_xM\rightarrow\Lambda^7(T_x^*M)$ and see that 
\[
\det(b_{\tf}|_x)^\frac19 = \left(1+\frac16|\tau|_\f^2a\right)^2 \det(b_{\f}|_x)^{\frac19}.
\]
Thus, the 3-form $\tf|_x$ is stable if and only if $a\neq-6|\tau|_\f^{-2}$. Moreover, we have that 
\[
g_{\tf}|_x = \det(b_{\tf}|_x)^{-\frac19} \,b_{\tf}|_x = \sum_{i=1}^3(e^i)^2+\left(1+\frac16|\tau|_\f^2a\right) \sum_{i=4}^7(e^i)^2. 
\]
Hence, $\tf|_x\in\Lambda^3_{\sst+}(T^*_xM)$ if and only if $a>-6|\tau|_\f^{-2}$, as we claimed. 

Now, assertions \ref{1}) and \ref{2}) follow immediately from the above discussion, 
while assertion \ref{3}) can be checked pointwise using the adapted basis for $\f$ we are considering. 

As for \ref{4}), we have 
\[
d*_{\tf}\tf = \left(1+\frac16|\tau|_\f^2a\right)\left(d*_\f\f -\frac{a}{6}\,d(\tau\W\tau)\right) = \left(1+\frac16|\tau|_\f^2a\right) \tau\W\tf,
\]
since $\tau\W\tau$ is closed. Thus, by the uniqueness of the 2-form $\widetilde\tau$ for which $d*_{\tf}\tf=\widetilde\tau\W\tf$, we obtain 
$\widetilde\tau=\left(1+\frac16|\tau|_\f^2a\right) \tau$. We can compute the Hodge dual of $\widetilde\tau$ as follows
\[
*_{\tf}\,\widetilde\tau  = -\widetilde\tau\W\tf = -\left(1+\frac16|\tau|_\f^2a\right)\tau\W\f = \left(1+\frac16|\tau|_\f^2a\right) *_\f\tau. 
\]
Consequently, we get
\[
|\widetilde\tau|_{\tf}^2\, dV_{\tf} = \widetilde\tau \W *_{\tf}\,\widetilde\tau = \left(1+\frac16|\tau|_\f^2a\right)^2 |\tau|_\f^2\,dV_\f = |\tau|_\f^2\, dV_{\tf}. 
\]

Finally, a pointwise computation as in \ref{3}) allows us to show that $*_{\tf}(\widetilde\tau\W\widetilde\tau) = *_\f(\tau\W\tau)$. Using this, it is straightforward to check that $\tf$ is ERP. 
Now, the integrable subbundles $\widetilde{P}$ and $\widetilde{Q}$ determined by the ERP closed $\G_2$-structure $\tf$ coincide with those determined by $\f$.  
Consequently, by point \ref{vi}) of Proposition \ref{ERPprop} and point \ref{1}) above, we have
\[
\mathrm{Ric}(g_{\tf}) = -\frac{1}{6}|\tau|_\f^2\,g_{\tf}|_{\widetilde{P}} = -\frac{1}{6}|\tau|_\f^2\,g_{\tf}|_{P} =  -\frac{1}{6}|\tau|_\f^2\,g_\f|_{P} = \mathrm{Ric}(g_\f).   
\]
\end{proof}

We are now ready to prove the main theorem of this section. 
\begin{proof}[Proof of Theorem \ref{MainThm}]
Consider the closed 3-form $\f(t)=\f + f(t)d\tau$, where $f$ is a real valued smooth function such that $f(0)=0$. 
Let $t$ be small enough so that $f(t)>-6|\tau|_\f^{-2}$. Then, by Proposition \ref{PrepProp} the 3-form $\f(t)$ defines a closed $\G_2$-structure with intrinsic torsion form 
$\tau(t) =  \left(1+\frac16|\tau|_\f^2f(t)\right)\tau$. Now, the Laplacian flow equation $\frac{\partial}{\partial t}\f(t)=\Delta_{\f(t)}\f(t)$ reads
\[
\frac{d}{dt}f(t)\, d\tau = \left(1+\frac16|\tau|_\f^2f(t)\right) d\tau.
\] 
Consequently, $\f(t)$ is the solution of the Laplacian flow starting from $\f$ at $t=0$ if and only if $f(t)$ solves the Cauchy problem
\[
\begin{cases}
\frac{d}{dt}f(t) = 1+\frac16|\tau|_\f^2f(t),\\
f(0)=0. 
\end{cases}
\]
Thus, we have 
\[
f(t) = \frac{6}{|\tau|_\f^2}\left(\exp\left(\frac{|\tau|_\f^2}{6}\,t\right)-1\right), 
\]
whence we see that $f(t)$ is defined for all $t\in\R$ and that it satisfies the condition $f(t)>-6|\tau|_\f^{-2}$. 
The second part of the theorem follows immediately from points \ref{4}) and \ref{5}) of Proposition \ref{PrepProp}. 
In particular, the evolution equation of the metric $g_{\f(t)}$ can be obtained starting from equation \eqref{metricevol} and using the identity 
$\frac{1}{4}\,\jf(*_\f(\tau\W\tau)) = 3\,\mathrm{Ric}(g_\f)$ (cf.~point \ref{vi}) of Proposition \ref{ERPprop}). 
\end{proof}

When $M$ is compact, backward uniqueness and real analyticity of the solution of the Laplacian flow (cf.~\cite[Thm.~1.4]{LotWei1} and \cite{LotWei3})  
together with Theorem \ref{MainThm} imply the following. 
\begin{corollary}
Let $\f(t)$ be the solution of the Laplacian flow \eqref{LapFlow} on a compact 7-manifold $M$ and assume that $\f(0)$ is not ERP. 
Then, $\f(t)$ cannot become ERP in finite time. 
\end{corollary}

Furthermore, from point \ref{4}) of Proposition \ref{ERPvariation} and Remark \ref{remnormdtau} we see that the velocity of the flow is constant for all $t\in\R$:
\[
\left|\Delta_{\f(t)}\f(t)\right|_{\f(t)} = |d\tau(t)|_{\f(t)} = \frac{1}{\sqrt{6}}|\tau|^2_{\f}.
\]

Since the solution $\f(t)$ of the Laplacian flow starting from an ERP closed $\G_2$-structure $\f$ exists for all real times,  
and since the Ricci tensor of the corresponding metric $g_{\f(t)}$ does not evolve along the flow, 
it is natural to ask whether $\f(t)$ is self-similar or, equivalently, if $\f$ is a steady Laplacian soliton. 
It follows from \cite[Cor.~1]{Lin} that the answer is negative when the manifold $M$ is compact (see also \cite[Prop.~9.5]{LotWei1}). In detail:
\begin{theorem}[\cite{Lin}]
Let $M$ be a compact 7-manifold. Then, the only steady Laplacian solitons on $M$ are given by torsion-free $\G_2$-structures.
\end{theorem}

On the other hand, it was recently shown in \cite{LaNi} that any left-invariant ERP closed $\G_2$-structure on a non-compact Lie group  is a steady Laplacian soliton.  
The converse of this result does not hold, as there are examples of left-invariant steady Laplacian solitons that are not ERP, see \cite{FiRa3}. 

\section{Asymptotic behaviour of the ERP solution}\label{AsBe}
In this section, we assume that the 7-manifold $M$ is compact and we investigate the behaviour of the ERP solution when $t\rightarrow\pm\infty$. 

Using point \ref{2}) of Proposition \ref{PrepProp}, we obtain the following expression for the volume form induced by $\f(t)$ 
\[
dV_{\f(t)} = \exp\left(\frac{|\tau|^2_\f}{3}\,t\right) dV_\f. 
\]
Consequently, the total volume of the compact 7-manifold $M$ with respect to the metric $g_{\f(t)}$ is
\[
\mathrm{Vol}_{g_{\f(t)}}(M) = \int_M dV_{\f(t)} =   \exp\left(\frac{|\tau|^2_\f}{3}\,t\right) \mathrm{Vol}_{g_\f}(M), 
\]
whence we easily see that it increases without bound as $t$ goes to $+\infty$, while it shrinks as $t$ goes to $-\infty$:
\[
\lim_{t\rightarrow+\infty}\mathrm{Vol}_{g_{\f(t)}}(M)  = +\infty,\qquad \lim_{t\rightarrow-\infty}\mathrm{Vol}_{g_{\f(t)}}(M)  = 0.
\]

We now study the behaviour of the associative $P$-leaves and coassociative $Q$-leaves along the flow. 
As the integrable subbundles $P(t)$ and $Q(t)$ determined by the ERP closed $\G_2$-structure $\f(t)$ coincide with the subbundles $P$ and $Q$ determined by $\f=\f(0)$, 
at each time $t\in\R$ we can endow the $P$-leaves and the $Q$-leaves with the Riemannian metric induced by $g_{\f(t)}$. 
We denote the corresponding Riemannian volume forms by $dV_P(t)$ and $dV_Q(t)$, respectively, and we let $dV_P\coloneqq dV_P(0),~dV_Q\coloneqq dV_Q(0)$. 

Consider an oriented $P$-leaf $L_P \hookrightarrow M$.  
By point \ref{v}) of Proposition \ref{ERPprop}, we know that the volume form $dV_P(t)$ on $L_P$ coincides with the 
restriction of the closed 3-form $-|\tau(t)|_{\f(t)}^{-2}*_{\f(t)}(\tau(t)\W\tau(t))$. Such a volume form is constant along the flow, as 
\[
dV_{P}  = \left.-|\tau|_\f^{-2}*_\f(\tau\W\tau)\right|_{L_P} = dV_{P}(t),
\]
by point \ref{4}) of Proposition \ref{PrepProp}. 
Using the same result, we see that the volume form $dV_{Q}(t)$ of an oriented $Q$-leaf $L_Q\hookrightarrow M$ is given by
\begin{eqnarray*}
dV_{Q}(t)	&=&	\left.-|\tau(t)|^{-2}_{\f(t)}(\tau(t)\W\tau(t))\right|_{L_Q} = \left.-|\tau|^{-2}_\f \exp\left(\frac{|\tau|^2_\f}{3}t\right)(\tau\W\tau)\right|_{L_Q}\\
		&=& \exp\left(\frac{|\tau|^2_\f}{3}t\right)dV_Q.
\end{eqnarray*}

It is now immediate to prove the following. 
\begin{proposition}
Let $\f(t)$ be the solution of the Laplacian flow starting from an ERP closed $\G_2$-structure on a compact 7-manifold $M.$ 
Then, the volume of the $P$-leaves is constant along the flow. Moreover:
\begin{enumerate}[1.]
\item when $t\rightarrow+\infty$, the volume of the $P$-leaves goes to zero relative to the volume of the manifold, while the volume of the $Q$-leaves 
and the volume of the manifold $M$ grow at the same rate; 
\item when $t\rightarrow-\infty$,  
the volume of the $Q$-leaves and the volume of the manifold $M$ tend to zero at the same rate. 
\end{enumerate}
\end{proposition}

\begin{remark}
Rescaling the metric $g_{\f(t)}$ as $\exp\left(-\frac{|\tau|^2_\f}{6}t\right)\,g_{\f(t)}$ shows that the volume of the $P$-leaves goes to zero as $t\rightarrow+\infty$, that is, the $P$-leaves collapse as $t\rightarrow+\infty$. 
\end{remark}

\section{Examples}\label{ExSect}
In this section, we review two examples of 7-manifolds endowed with an ERP closed $\G_2$-structure obtained in \cite{Bry,Lau2}, 
and we discuss some related results. 
Further examples are considered in \cite{Bal,LaNi,LaNi2}. 

We begin with the example obtained by Bryant in \cite[Ex.~1]{Bry}. 
It consists of an invariant ERP closed $\G_2$-structure on the non-compact homogeneous space $\SL(2,\C)\ltimes\C^2 / \SU(2)$. 
For the sake of convenience, we give the following alternative description (cf.~\cite[Sect.~6.3]{ClIv2} and \cite[Ex.~4.13]{Lau2}).

\begin{example}\label{BryEx} 
Let $\frr=\langle e_1,e_2,e_3\rangle$ be the three-dimensional non-unimodular solvable Lie algebra with non-zero Lie brackets 
\[
[e_1,e_2] = e_2,\quad [e_1,e_3]=e_3. 
\]
Consider the abelian Lie algebra $\R^4 = \langle e_4,e_5,e_6,e_7\rangle$, and the Lie algebra homomorphism 
$\mu:\frr\rightarrow\mathrm{Der}(\R^4)\cong\mathfrak{gl}(4,\R)$ defined as follows
\[
\mu(e_1) = \mathrm{diag}\left(-\frac12,-\frac12,\frac12,\frac12\right),~
\mu(e_2) = \left(\begin{array}{cccc} 0&0&0&0\\ 0&0&0&0\\ 0&-1&0&0\\ -1&0&0&0\\  \end{array}\right),~
\mu(e_3) = \left(\begin{array}{cccc} 0&0&0&0\\ 0&0&0&0\\ -1&0&0&0\\ 0&1&0&0\\  \end{array}\right). 
\]
The semidirect product  $\frs\coloneqq\frr\ltimes_\mu\R^4$ is a seven-dimensional non-unimodular completely solvable Lie algebra. 
Denoted by $(e^1,\ldots,e^7)$ the dual basis of $(e_1,\ldots,e_7)$, the structure equations $\left(de^i\right)_{i=1,\ldots,7}$ of $\frs$ are given by 
\[
\left(0, -e^{12}, -e^{13}, \frac12\,e^{14}, \frac12\,e^{15}, -\frac12\,e^{16}+e^{25}+e^{34},  -\frac12\,e^{17}+e^{24}-e^{35}\right).
\]
It is now straightforward to check that the 3-form $\f$ given by \eqref{G2adapted} defines an ERP closed $\G_2$-structure on $\frs$ with intrinsic torsion form 
\[
\tau = 3\,e^{45}-3\,e^{67}.
\]
From this we see that $|\tau|^2_\f=18$, $P=\frr$, and $Q=\R^4$. 
Notice that $\f$ is exact:
\[
\f = d\left( -\frac12 e^{23} +e^{45} -e^{67}\right). 
\]
Left-multiplication allows to extend the 3-form $\f$ to a left-invariant one, say $\widetilde\f$, on the simply connected solvable Lie group $\mathrm{S}$ with Lie algebra $\frs$. 
Bryant's example is then described by the pair $(\mathrm{S},\widetilde\f)$. 
\end{example}

\begin{remark}\ 
Albeit the Lie algebra $\frs$ is not unimodular, the corresponding simply connected solvable Lie group $\mathrm{S}$ 
admits a compact quotient, as it is acted on by a torsion-free discrete subgroup $\Gamma \subset \mathrm{Aut}(\mathrm{S},\widetilde\f)$ 
(cf.~\cite[Ex.~1]{Bry} and \cite[Remark 4.14]{Lau2}).  
This gives rise to a compact locally homogeneous  example of ERP closed $\G_2$-structure on $\Gamma\backslash\mathrm{S}$. 
\end{remark}

As observed by Cleyton and Ivanov \cite{ClIv2}, in the above example the intrinsic torsion form $\tau$ is parallel with respect 
to the canonical $\G_2$-connection $\overline{\nabla}$. More generally, the following holds.  
\begin{theorem}[\cite{ClIv2}] 
Let $M$ be a 7-manifold endowed with a closed $\G_2$-structure $\f$ and let $\overline{\nabla}$ be the corresponding canonical $\G_2$-connection. 
If the intrinsic torsion form $\tau$ of $\f$ is parallel with respect to $\overline{\nabla}$, then $\f$ is ERP and $(M,g_\f)$ is locally isometric to Bryant's example. 
\end{theorem}

We now describe a one-parameter family of pairwise non-isomorphic solvable Lie algebras admitting an ERP closed $\G_2$-structure and including the above example. 
\begin{example}\label{1ParEx}
Let us consider the one-parameter family of three-dimensional non-unimodular solvable Lie algebras $\frr_\eta = \langle e_1,e_2,e_3\rangle$ with non-zero Lie brackets
\[
[e_1,e_2] = e_2 + \eta\,e_3,\quad [e_1,e_3] = -\eta\,e_2 + e_3, \, \, \eta\in\R. 
\]
Notice that $\frr_{\sst0} = \frr$, while the number $1+\eta^2$ provides a complete isomorphism invariant for $\frr_\eta$ when $\eta\neq0$ 
(see \cite[Lemma 4.10]{Mil}).
Moreover, it is possible to check that the Ricci endomorphism of the inner product $g=(e^1)^2+(e^2)^2+(e^3)^2$ on $\frr_\eta$ is diagonal with eigenvalue $-2$ of multiplicity three 
(cf.~\cite[Thm.~4.11]{Mil}). 

We let $\frs_\eta\coloneqq \frr_\eta\ltimes_{\mu_\eta}\R^4$, where $\R^4 = \langle e_4,e_5,e_6,e_7\rangle$ is the abelian Lie algebra, and the Lie algebra homomorphism 
$\mu_\eta:\frr_\eta\rightarrow\mathrm{Der}(\R^4)\cong\mathfrak{gl}(4,\R)$ is defined as follows
\[
\mu_\eta(e_1) = \left(\begin{array}{cccc} {\scriptstyle-}\frac12&0&0&0\\ 0&{\scriptstyle-}\frac12&0&0\\ 0&0&\frac12&\eta\\ 0&0&{\scriptstyle-}\eta&\frac12\\  \end{array}\right),\,
\mu_\eta(e_2) = \left(\begin{array}{cccc} 0&0&0&0\\ 0&0&0&0\\ 0&{\scriptstyle-}1&0&0\\ {\scriptstyle-}1&0&0&0\\  \end{array}\right),\,
\mu_\eta(e_3) =  \left(\begin{array}{cccc} 0&0&0&0\\ 0&0&0&0\\ {\scriptstyle-}1&0&0&0\\ 0&1&0&0\\  \end{array}\right). 
\]
When $\eta\neq0$, the seven-dimensional non-unimodular Lie algebra $\frs_\eta$ is solvable but not completely solvable,   
as the complex numbers $1\pm i\eta, \frac12\pm i\eta$ are eigenvalues of $\mathrm{ad}_{e_1}$. 
Furthermore, $\frs_{\sst0}$ coincides with the Lie algebra $\frs$ considered in Example \ref{BryEx}. 
 
Now, the 3-form $\f$ given by \eqref{G2adapted} is exact, and it defines an ERP closed $\G_2$-structure on $\frs_\eta$ with intrinsic torsion form 
$\tau = 3\,e^{45}-3\,e^{67}$. 
By left-multiplication, we can extend $\f$ to a left-invariant ERP closed $\G_2$-structure $\widetilde\f_\eta$ 
on the simply connected solvable Lie group $\mathrm{S}_\eta$ with Lie algebra $\frs_\eta$.  
Clearly, the Lie groups $\mathrm{S}_\eta$ are pairwise non-isomorphic for all $\eta\geq0$.
However, it is possible to check that the intrinsic torsion form of the ERP closed $\G_2$-structure on $\mathrm{S}_\eta$ is parallel with respect to the canonical $\G_2$-connection. 
Hence, by the result of Cleyton and Ivanov recalled above, $(\mathrm{S}_\eta,g_{\widetilde\f_\eta})$ is locally isometric to Bryant's example.
\end{example}

\begin{remark}
Although the simply connected solvable Lie groups $\mathrm{S}_\eta$ are pairwise non-isomorphic for different values of $\eta\geq0$, 
it is possible to show that the $\G_2$-structures $(\mathrm{S}_\eta,\widetilde\f_\eta)$ are pairwise equivalent.  
This has been proved in the recent work \cite{LaNi}. 
\end{remark}

The next example was obtained by Lauret in \cite{Lau2}. As shown in the same paper, it is not equivalent to Bryant's example (cf.~\cite[Rem.~4.15]{Lau2}). 
\begin{example}\label{Lauretexample}
Consider the abelian Lie algebras $\R^3=\langle e_1,e_2,e_3\rangle$ and $\R^4=\langle e_4,e_5,e_6,e_7\rangle$, and define the seven-dimensional unimodular solvable Lie algebra 
$\fru$ as the semidirect product $\R^3\ltimes_\mu\R^4$, with $\mu:\R^3\rightarrow\mathrm{Der}(\R^4)\cong\mathfrak{gl}(4,\R)$ given by
\[
\mu(e_1) = \mathrm{diag}(1,1,-1,-1),\quad \mu(e_2) = \mathrm{diag}(1,-1,1,-1),\quad \mu(e_3) = \mathrm{diag}(1,-1,-1,1).
\]
The structure equations of $\fru$ can be written with respect to the dual basis $(e^1,\ldots,e^7)$ of $(e_1,\ldots,e_7)$ as follows
\[
\left(0, 0, 0,  -e^{14}-e^{24}-e^{34}, -e^{15}+e^{25}+e^{35},  e^{16}-e^{26}+e^{36}, e^{17}+e^{27}-e^{37}   \right).
\]
The 3-form $\f$ given by \eqref{G2adapted} defines an ERP closed $\G_2$-structure on $\fru$ with intrinsic torsion form 
\[
\tau = -2\,e^{45}+2\,e^{67}-2\,e^{46}-2\,e^{57}+2\,e^{47}-2\,e^{56}. 
\]
Notice that $|\tau|_\f^2=24$, and that $P$ and $Q$ coincide with $\R^3$ and $\R^4$, respectively. 
\end{example}

\begin{remark}
As shown in \cite{Lau2}, in both Examples \ref{BryEx} and \ref{Lauretexample} the left-invariant ERP closed $\G_2$-structure is a steady Laplacian soliton on the corresponding 
non-compact simply connected solvable Lie group. By \cite[Cor.~1]{Lin}, it cannot descend to an invariant steady Laplacian soliton on any compact quotient.
\end{remark}

In the next theorem, we show that the Lie algebra $\fru$ described in the previous example is the unique unimodular Lie algebra admitting ERP closed $\G_2$-structures 
up to isomorphism. 

\begin{theorem}\label{UnimUniq}
Let $\G$ be a unimodular Lie group endowed with a left-invariant  ERP closed $\G_2$-structure. 
Then, its Lie algebra is isomorphic to the Lie algebra $\fru$ described in Example \ref{Lauretexample}.  
\end{theorem} 
\begin{proof} 
Consider the Lie algebra $\frg=\mathrm{Lie}(\G)$ endowed with the ERP closed $\G_2$-structure $\f$ corresponding to the left-invariant one on $\G$, 
and let $\frg=P\oplus Q$ be the induced $g_\f$-orthogonal decomposition of $\frg\cong T_{\sst 1_\G}\G$.
Since $\frg$ is unimodular and the Ricci tensor of $g_\f$ is non-positive, the nilradical  $\frn$ of $\frg$ is abelian by \cite[Cor.~1]{Dot1}.  
Moreover, by \cite[Lemma 1]{Dot1} we have $\Ric(g_\f) \vert_{\frn} =0$. 
Consequently, point \ref{vi}) of Proposition \ref{ERPprop} implies that the nilradical $\frn$ is contained in $Q$.  

To prove the assertion, we consider all possible seven-dimensional unimodular Lie algebras with abelian nilradical of dimension at most four, and we show that $\fru$ is the only one 
admitting an ERP closed $\G_2$-structure up to isomorphism. We shall deal with the cases $\frg$ solvable and $\frg$ non-solvable separately. 

If $\frg$ is solvable, by \cite[Thm.~1]{Sno} 
\[
\dim(\frg) = \dim(\frn) + k,  
\]
with $k \leq \dim(\frn) - \dim\left(\left[\frn,\frn\right]\right)$. As $\frn$ is abelian, we have 
\[
7 = \dim(\frn) + k, \quad k \leq \dim(\frn) \leq \dim(Q) =4,
\]
whence $\frn = Q$. Thus, the nilradical of $\frg$ is four-dimensional and abelian.  
Moreover, since $\frg$ is solvable, the dimension of its center $\frz(\frg)$ must satisfy $\dim(\frz(\frg))\leq 2\dim(\frn)-\dim(\frg)=1$ (see e.g.~\cite{Mub1}). 
We now claim that $\frg$ is not decomposable. 
Indeed, otherwise we could write $\frg = \frs_1 \oplus \frs_2$ for some solvable unimodular ideals $\frs_1,\frs_2\subseteq\frg$. 
If $\dim(\frs_1)=2$, then necessarily $\frs_1\cong\R^2$ and $\frz(\frg)$ would have dimension at least two, a contradiction. 
If $\dim(\frs_1)=3$, then $\frs_1$ is isomorphic either to the Lie algebra $\fre(1,1)$ of the group of rigid motions of the Minkowski 2-space or to the Lie algebra $\fre(2)$ of the group of rigid motions 
of the Euclidean 2-space (see, for instance, \cite[Thm.~1.1]{ABDO}). 
Since both these Lie algebras have two-dimensional abelian nilradical, it follows that the unimodular solvable Lie algebra $\frs_2$ must be four-dimensional with two-dimensional abelian nilradical. 
However, there are no four-dimensional Lie algebras satisfying these properties by \cite[Thm.~1.5]{ABDO}. 
Finally, if $\dim(\frs_1)=1$, then $\frs_1\cong\R$ and $\frs_2$ must be six-dimensional with three-dimensional nilradical.  
Any such $\frs_2$ must be decomposable by \cite[Thm.~3]{NdWi}. We can then conclude as in the previous cases.  
Therefore, $\frg$ is indecomposable.  

Indecomposable seven-dimensional solvable Lie algebras with an abelian nilradical of dimension four were classified in \cite{HiTh}. 
A scan of all possibilities allows to conclude that only three unimodular Lie algebras of this type occur up to isomorphism. 
One is isomorphic to the Lie algebra $\fru$ described in Example \ref{Lauretexample}, thus it admits an ERP closed $\G_2$-structure. 
The remaining ones correspond to the fourth and the fifth pencil in \cite[Prop.~5.4]{HiTh}. 
Their structure equations with respect to a suitable basis $(e^1,\ldots,e^7)$ are the following:
\begin{equation}\label{1case}
\left\{ 
\renewcommand\arraystretch{1.3}
\begin{array}{rcl}
de^i 	&=&	0,\quad i=1,2,3,\\
de^4	&=& \alpha\,e^{14}-e^{15}+\gamma\,e^{24},\\
de^5	&=& e^{14}+\alpha\,e^{15}+\gamma\,e^{25},\\
de^6	&=& -\alpha\,e^{16}-\beta\,e^{17}-\gamma\,e^{26}-\rho\,e^{27}-\sigma\,e^{37},\\
de^7	&=&	\beta\,e^{16}-\alpha\,e^{17}+\rho\,e^{26}-\gamma\,e^{27}+\sigma\,e^{36}, 
\end{array}
\right.
\end{equation}
where $\alpha,\beta,\rho\in\R$ and $\gamma,\sigma\in\R\smallsetminus\{0\}$, and
\begin{equation}\label{2case}
\left\{ 
\renewcommand\arraystretch{1.3}
\begin{array}{rcl}
de^i 	&=&	0,\quad i=1,2,3,\\
de^4	&=& \frac12\alpha\,e^{14}-e^{15}+\frac12\beta\,e^{24},\\
de^5	&=& e^{14}+\frac12\alpha\,e^{15}+\frac12\beta\,e^{25},\\
de^6	&=& -e^{36},\\
de^7	&=&	-\alpha\,e^{17}-\beta\,e^{27}+e^{37}, 
\end{array}
\right.
\end{equation}
where $\alpha\in\R$ and $\beta\in\R\smallsetminus\{0\}$.  
We now show that none of these Lie algebras admits closed $\G_2$-structures. 
The generic closed 3-form $\phi$ on the Lie algebra with structure equations \eqref{1case} is given by
\begin{eqnarray*}
\phi	&=&	\phi_{123}e^{123} + \phi_{124}e^{124} + \phi_{125}e^{125} + \phi_{135}e^{135} + \phi_{136}e^{136} + \phi_{137}e^{137} + \phi_{147}e^{147} + \phi_{157}e^{157}  \\
	& & +\frac{\rho}{\sigma} \phi_{357} e^{257} +\frac{\rho}{\sigma} \phi_{347} e^{247} +\frac{\alpha}{\gamma} \phi_{245} e^{145} +\left(\sigma \phi_{147}-\beta\phi_{347}\right) e^{356} + 
		\left(\beta\phi_{357}-\sigma\phi_{157}\right)e^{346} \\
	& & +\left(\alpha\phi_{235} - \gamma\phi_{135} \right) e^{234} +\rho\left(\phi_{147} - \frac{\beta}{\sigma}\phi_{347}\right) e^{256} +\rho\left(\frac{\beta}{\sigma}\phi_{357} - \phi_{157} \right) e^{246} + \phi_{245}e^{245}  \\
	& & +\frac{\alpha^2\phi_{235} -\alpha\gamma\phi_{135} +\phi_{235}}{\gamma}e^{134} + \phi_{235}e^{235} + \phi_{236}e^{236} + \phi_{237}e^{237}  +  \phi_{347}e^{347} + \phi_{357}e^{357}\\
	& & +\frac{\beta^2\phi_{357} -\beta\sigma\phi_{157} -\phi_{357}}{\sigma}e^{146} -\frac{\beta^2\phi_{347} -\beta\sigma\phi_{147} -\phi_{347}}{\sigma}  e^{156} +\frac{\alpha}{\gamma}\phi_{267}e^{167}+ \phi_{267}e^{267} \\
	& &  +\frac{\alpha\phi_{236} -\beta\phi_{237} -\gamma\phi_{136}+\rho\phi_{137} }{\sigma} e^{127} - \frac{\alpha\phi_{237} +\beta\phi_{236} -\gamma\phi_{137}-\rho\phi_{136}}{\sigma} e^{126},  
\end{eqnarray*}
where $\phi_{ijk}\in\R$. A simple computation shows that 
\[
b_\phi(e_6,e_6) = -\phi_{267}\left(\phi_{147}\phi_{357}-\phi_{157}\phi_{347} \right) e^{1234567} = -b_\phi(e_7,e_7),
\]
thus $\phi$ cannot define a $\G_2$-structure. 

As for the Lie algebra with structure equations \eqref{2case}, the generic closed 3-form has the following expression
\begin{eqnarray*}
\phi	&=& \phi_{123}e^{123}+\ \phi_{124}e^{124}+\phi_{125}e^{125} +\phi_{135}e^{135} + \phi_{136}e^{136} + \phi_{137}e^{137} + \phi_{235}e^{235} + \phi_{236}e^{236} \\
	& & + \phi_{237}e^{237} + \phi_{245}e^{245}+\phi_{267}e^{267} +\phi_{346}e^{346} + \phi_{347}e^{347} + \phi_{356}e^{356}  + \phi_{357}e^{357}  \\
	& & +\frac12\left(\alpha\phi_{235}-\beta\phi_{135}\right) e^{234} + \left(\phi_{346} -\frac{\alpha}{2}\phi_{356} \right) e^{156} - \left(\phi_{347} +\frac{\alpha}{2}\phi_{357} \right) e^{157}  +\alpha\phi_{267}e^{167}   \\
	& & - \left(\phi_{356} +\frac{\alpha}{2}\phi_{346} \right) e^{146} + \left(\phi_{357} -\frac{\alpha}{2}\phi_{347} \right) e^{147} +\left(\alpha\phi_{237} - \beta\phi_{137}\right) e^{127} +\alpha\phi_{245}e^{145}\\
	& & -\frac{\beta}{2}\left(\phi_{347}e^{247} +\phi_{346}e^{246}  +\phi_{357}e^{257}  +\phi_{356}e^{256}\right) +\left(\frac{\alpha^2}{2}\phi_{235}-\frac{\alpha}{2}\phi_{135} +2\phi_{235}\right)e^{134},
\end{eqnarray*}
with $\phi_{ijk}\in\R$, and we immediately see that also in this case $b_\phi(e_6,e_6) = -b_\phi(e_7,e_7)$.

We now focus on the case when $\frg$ is unimodular and non-solvable. 
By the classification in \cite{FiRa2}, $\frg$ is isomorphic to one of the following 
\begin{eqnarray*}
&\left(-e^{23},-2e^{12},2e^{13},0,-e^{45},\frac{1}{2}e^{46}-e^{47},\frac{1}{2}e^{47}\right);&  \\ 
&\left(-e^{23},-2e^{12},2e^{13},0,-e^{45},-\alpha\,e^{46},(1+\alpha)\,e^{47}\right), ~ -1<\alpha\leq-\frac12;& \\ 
&\left(-e^{23},-2e^{12},2e^{13},0, -\alpha\,e^{45},\frac{\alpha}{2}\,e^{46}-e^{47},e^{46}+\frac{\alpha}{2}\,e^{47}\right),~\alpha>0;& \\ 
&\left(-e^{23},-2e^{12},2e^{13},-e^{14}-e^{25}-e^{47},e^{15}-e^{34}-e^{57},2e^{67},0\right).& 
\end{eqnarray*}
All of these Lie algebras have a three-dimensional abelian nilradical $\frn$. 
Since the 3-form defining an ERP closed $\G_2$-structure vanishes on $Q$, it must vanish on $\frn$. 
The first three Lie algebras in the above list do not admit any stable closed 3-form satisfying this condition.  
Indeed, in each case the abelian nilradical is given by $\frn=\langle e_5,e_6,e_7\rangle$ and imposing that the generic closed 3-form $\phi$ satisfies 
$\phi(e_5,e_6,e_7)=0$ gives $b_\phi(e_i,e_i)=0$, for $i=5,6,7$. 
We are then left with the last Lie algebra, whose nilradical is $\frn=\langle e_4,e_5,e_6\rangle$.  
Let us consider the generic closed 3-form 
\begin{eqnarray*}
\phi	&=&  \phi_{123}e^{123}-3\phi_{247}e^{124}-\phi_{234}e^{125}+\phi_{267}e^{126}+\phi_{127}e^{127}-\phi_{235}e^{134}+3\phi_{357}e^{135} \\
	&  &	-\phi_{367}e^{136}+\phi_{137}e^{137}+\phi_{467}e^{146} +\left(\phi_{257}-\phi_{234}\right) e^{147} -\phi_{567}e^{156} -\left(\phi_{235}+\phi_{347}\right)e^{157}\\
	&  &	+2\phi_{236}e^{167}+\phi_{234}e^{234}+\phi_{235}e^{235}+\phi_{236}e^{236}+\phi_{237}e^{237}+\phi_{247}e^{247}+\phi_{467}e^{256}\\
	&  &	+\phi_{257}e^{257}+\phi_{267}e^{267}+\phi_{567}e^{346}+\phi_{347}e^{347}+\phi_{357}e^{357}+\phi_{367}e^{367}+\phi_{456}e^{456}\\
	&  &	+\phi_{457}e^{457}+\phi_{467}e^{467}+\phi_{567}e^{567},
\end{eqnarray*}
where $\phi_{ijk}\in\R$. The condition $\phi(e_4,e_5,e_6)=0$ gives $\phi_{456}=0$. 
Up to a basis change, we may assume that $Q=\langle e_4,e_5,e_6,v\rangle$ with $v = v_1e_1+v_2e_2+v_3e_3+v_7e_7$ for some real numbers $v_1,v_2,v_3,v_7$.  
Now, if we consider the equation $0=\phi(e_4,e_5,v) = \phi_{457}\,v_7$, we see that necessarily $v_7=0$, otherwise $b_\phi(e_4,e_4)=0$ and $\phi$ would not be stable. 
An inspection of the equations $\phi(e_4,e_6,v) = 0 = \phi(e_5,e_6,v)$ when $\phi$ is stable gives the following possibilities
\begin{enumerate}[$\cdot$]
\item $\phi_{467} \neq 0$, $\phi_{567} \neq 0$, and $Q = \left\langle e_4, e_5, e_6, e_1 + \frac{\phi_{567}}{\phi_{467}} e_2 -  \frac{\phi_{467}}{\phi_{567}} e_3\right\rangle$;
\item $\phi_{467} =0$, $\phi_{567} \neq 0$, and $Q = \langle e_4, e_5, e_6, e_2\rangle$;
\item $\phi_{467} \neq 0$, $\phi_{567}= 0$, and $Q = \langle e_4, e_5, e_6, e_3\rangle$.
\end{enumerate} 
Now, if $\tau$ is the intrinsic torsion form of an ERP closed $\G_2$-structure, then the 4-form $\tau\W\tau\in\Lambda^4(Q^*)$ must be closed and simple. 
However, in none of the above cases there exist closed simple 4-forms on $Q$. 
\end{proof}

\medskip\noindent
{{\bf Acknowledgements.} The authors would like to thank Jason Lotay for useful comments and suggestions, and Fabio Podest\`a for useful conversations. 
Most of this work was done when A.R.~was a postdoctoral fellow at the Department of Mathematics and Computer Science ``U.~Dini'' of the 
Universit\`a degli Studi di Firenze. He is grateful to the department for hospitality.

\end{document}